\documentclass[12pt]{article}
\usepackage[english]{babel}
\usepackage{amsmath,amssymb,amsthm,exscale,amsopn,mathtools,fullpage}

\newtheorem{theorem}{Theorem}[section]
\newtheorem{lemma}[theorem]{Lemma}
\newtheorem{proposition}[theorem]{Proposition}
\newtheorem{corollary}[theorem]{Corollary}

\newcommand{\supp}{\operatorname{supp}}

\newcommand{\GK}{\operatorname{GKdim}}

\newcommand{\gr}{\operatorname{gr}}

\newcommand{\s}{\subseteq}
\newcommand{\arr}{\xrightarrow}

\DeclareMathOperator{\HK}{HK}

\begin{document}

    \title{Irreducible representations of Hecke--Kiselman monoids }

    \author{Magdalena Wiertel}
    \date{}
    \maketitle

    \begin{abstract}
        Let $K[\HK_{\Theta}]$ denote the Hecke--Kiselman algebra of a finite oriented graph $\Theta$
        over an algebraically closed field $K$.
        All irreducible representations, and the corresponding
        maximal ideals of $K[\HK_{\Theta}]$, are characterized in case this algebra satisfies a
        polynomial identity. The latter condition corresponds to a simple condition that can be expressed in terms
        of the graph $\Theta$.
        The result shows a surprising similarity to the classical results on representations of finite
        semigroups; namely
        every representation either comes form an idempotent in the Hecke--Kiselman monoid $\HK_{\Theta}$
        (and hence it is $1$-dimensional),
        or it comes from certain semigroup of matrix type (which is an order in a completely $0$-simple
        semigroup over an infinite cyclic group).
        The case when $\Theta$ is an oriented cycle plays a crucial role; the prime spectrum of
        $K[\HK_{\Theta}]$ is completely characterized in this case.
    \end{abstract}

    \vspace{20pt}



\section{Introduction}

For any finite simple oriented graph $\Theta$ with $n$ vertices, a finitely
generated monoid $\HK_{\Theta}$ was introduced by Ganyushkin and
Mazorchuk in \cite{maz} as the monoid generated by $n$ idempotents
satisfying the following relations, depending on the graph
$\Theta$.
\begin{itemize}
    \item[(i)] $\HK_{\Theta}$ is generated by elements $ x_i=x_{i}^2$,
    where $1 \leqslant i \leqslant n$,
    \item[(ii)] if the vertices $i$, $j$ are not connected in  $\Theta$,
    then  $x_ix_j = x_jx_i$,
    \item[(iii)] if $i$, $j$ are connected by an arrow $i \to j$ in $\Theta$,
    then $x_ix_jx_i = x_jx_ix_j = x_ix_j$.
\end{itemize}
The definition is a generalization of two different classes,
namely $0$--Hecke and Kiselman's monoids. In particular,
Hecke--Kiselman monoids are natural homomorphic images of the
corresponding $0$--Hecke monoids (that is, monoids where relations
(iii) are replaced by $x_ix_jx_i=x_jx_ix_j$) that play an
important role in representation theory and algebraic
combinatorics. Therefore the investigation of the structure and
representations of Hecke--Kiselman monoids can shed a new light on
problems related to the latter class.

Several combinatorial and structural properties of Hecke--Kiselman
monoids have been already studied. Some aspects of representations
in the case of finite monoids were considered in \cite{maz} and
\cite{type A}. In particular, as an illustration of the
representation theory of finite $\mathcal{J}$--trivial monoids,
all simple modules were described in this case. Note also that
certain representations by matrices with integral coefficients
were used in \cite{aragona} to show that the monoid associated to
an oriented cycle is infinite. Generalizations of this
representation have been recently applied in \cite{Lebed} to
construct combinatorial and diagrammatic interpretations of
elements of this monoid.


The investigation of ring theoretic properties of the infinite
dimensional Hecke--Kiselman algebras was started in
\cite{mecel_okninski1}. Results obtained therein and in
\cite{semiprime, OW} suggest that the structure of the
Hecke--Kiselman monoid associated to the oriented cycle of length
$n$ plays a crucial role in understanding the structure of the
algebras $K[\HK_{\Theta}]$, over a field $K$, for arbitrary
oriented graphs $\Theta$.

In this paper we focus on classical representations of Hecke--Kiselman algebras.

The main part is devoted to irreducible representations and prime
spectrum of the algebra associated to an oriented cycle. Then
maximal ideals in the case of any PI Hecke--Kiselman algebra are
described. Roughly speaking, maximal ideals of $K[\HK_{\Theta}]$
come from maximal ideals of the algebras associated to connected
components of certain subgraph of $\Theta$ that are either
singletons or oriented cycles.

Our techniques rely on the surprising ideal chain and structures
of matrix type hidden in the monoids $C_n$ associated to oriented cycles, characterized in \cite{OW}. The structures
occurring in the quotients have the flavour of the principal
factors of a finite semigroup, thus we are able to build a class
of irreducible representations of $C_n$ from the representations
of these factors. It turns out that the remaining representations
arise from idempotents in $C_n$, in a way similar to that known
for the so called $\mathcal J$-trivial finite monoids, see
\cite{denton, steinberg}.

Let us introduce some notation and recall necessary results. $C_n$
is the Hecke--Kiselman monoid associated to the oriented cycle of
length $n$, that is the graph
$1\arr{}2\arr{}\cdots\arr{}n-1\arr{}n\arr{}1$. Denote the
corresponding generators of $C_n$ by $x_1,\ldots, x_n$. $K[C_n]$
stands for the associated monoid algebra over an algebraically
closed field $K$. In the results from \cite{OW} the family of the
words $x_nx_1\cdots x_{i}x_{n-1}\cdots x_{i+1}$, for $i=0,\ldots,
n-2$, plays a crucial role. We denote these words by $x_nq_i$, or
$s_i$, and by $\langle x_nq_i \rangle$ we understand the cyclic
monoid generated by $x_nq_i$. We recall the result characterizing
the structure of the Hecke--Kiselman monoid $C_n$, that will be
crucial in our approach to representation theory of $K[C_n]$,
proved in \cite{OW}, Corollary~3.25 and \cite{semiprime}, proof of
Theorem~1.1.

\begin{proposition}[\cite{OW}]\label{struktura}
    $C_n$ has a chain of ideals
    \begin{equation*}\emptyset =I_{n-2} \vartriangleleft  I_{n-3}\vartriangleleft \cdots \vartriangleleft I_0\vartriangleleft
    I_{-1}  \vartriangleleft C_n
    \end{equation*}
    given by $I_i=\{x\in C_n: C_nxC_n\cap \langle x_nq_i\rangle=\emptyset\}$ for $i\neq -1$ and
    $I_{-1}=I_0\cup C_n x_nq_{0}C_n$,
    \\with the following properties
    \begin{enumerate}
        \item for $i=0,\ldots, n-2$ there exist semigroups of matrix type
        $M_i=\mathcal{M}^0(S_i, A_i, B_i; P_i)$,
        such that $M_i\s I_{i-1}/I_i$, where $S_{i}$ is the cyclic semigroup generated
        by $x_nq_i$, $P_i$
        is a square symmetric matrix of size $B_i\times A_i$ and with coefficients in
        $\langle x_nq_i\rangle\cup\{\theta\}$;
        \item for $i=1,\ldots, n-2$ the sets $(I_{i-1}/I_i)\setminus M_i$ are finite;
        \item $I_{-1}/I_0=M_0$;
        \item $\tilde{M}_{n-2}=M_{n-2}\setminus \{\theta\}\vartriangleleft C_n$;
        \item  $C_n/I_{-1}$ is a finite semigroup;
        \item every $M_i$ is a two-sided ideal in $C_n/I_i$;
        \item $|A_i|=|B_i|= {n\choose i+1}$.
    \end{enumerate}
\end{proposition}

Semigroups of matrix type $M_i\s I_{i-1}/I_i$ as sets consist of
elements from $(I_{i-1}\setminus I_{i})\cup \{\theta\}$, where
$\theta$ is the zero element, and thus (without $\theta$) can be
treated as subsets of $C_n$, denoted also by $M_i$, abusing slightly the notation. We define
$M=\bigcup_{i=0}^{n-2}M_i$. It has been proved that
$C_n\setminus M$ is finite. Note that $M$ is not a subsemigroup of
$C_n$. By $K_0[M_i]$ we denote the contracted semigroup algebra
associated to $M_i$, understood as a subsemigroup of
$I_{i-1}/I_{i}$. Moreover $K[M_i]$ is the subspace spanned by
elements of $M_i$ treated as the set of words in $C_n$, if
unambiguous.

The paper is organized as follows. In Section~\ref{idempotents} we
revisit the idempotents in $C_n$, characterized in a more general
case in \cite{maz}, to place them in the ideal structure of the
monoid. Section~\ref{primespectrum} is devoted to the prime
spectrum of the algebra $K[C_n]$. The main result characterizing
irreducible representations of $K[C_n]$ is proved in Section~\ref{repC_n}, Theorem~\ref{irred main}. Then we calculate the latter representations
(Section~\ref{repM}) and illustrate the results in the simplest case of the
monoid associated to the graph that is a cycle of length $3$. In the last section
we describe maximal ideals of arbitrary Hecke--Kiselman algebras
which satisfy a polynomial identity.

\section{Idempotents in $C_n$}\label{idempotents}
We are going to characterize all idempotents in the
Hecke--Kiselman monoid associated to the cycle of length $n$. This
will be an intermediate step in our approach to a description of
all irreducible representations of this monoid. Firstly, let us
recall a more general theorem about idempotents in Hecke--Kiselman
monoids associated to any oriented graph.
\begin{theorem}[\cite{type A}, Theorem 3.6]\label{subgraphs}
    There is a bijection between idempotents in $\HK_{\Gamma}$ and full subgraphs
    of $\Gamma$ which do not contain any oriented cycles.
\end{theorem}
The bijection from the above theorem follows from the fact that
if $\Gamma$ is an acyclic graph, then $\HK_{\Gamma}$ is a quotient
of a Kiselman's monoid. Idempotents in the latter monoid have been
explicitly described in the paper \cite{OnK}. More precisely, the
following theorem describes idempotents in the Hecke--Kiselman
monoid associated to any acyclic graph.
\begin{theorem}[\cite{maz}, Theorem 1 (iii)]\label{id}
    Assume that $\Gamma$ is a simple oriented graph with the set of
    vertices $\{1, \ldots, n\}$ such that if $i\arr{} j$ in $\Gamma$
    then $i<j$, i.e. $\Gamma$ is acyclic. Then $\HK_{\Gamma}$ has
    exactly $2^n$ idempotents. More precisely every idempotent is of
    the form $e_X$, where for every $X\subseteq\{1,\ldots, n\}$ of the
    form $X=\{i_1, \ldots, i_j \}$ with $i_1<i_2<\cdots<i_j$ we denote
    by $e_X$ the element $x_{i_1}\cdots x_{i_j}$ (for $X=\emptyset$ we
    set $e_X=1$).
\end{theorem}

Let $C_n$ be the Hecke--Kiselman monoid associated to the cycle of length $n$, that is
the graph $1\arr{}2\arr{}\cdots\arr{}n\arr{}1$ and denote by $\{x_1,\ldots, x_n\}$ the corresponding
set of generators in $C_n$.\\
We use Theorem \ref{id} to list all idempotents in $C_n$. First
note that every proper full subgraph of the cycle is acyclic.
Consider such a subgraph $\Gamma'$ and denote
$V(\Gamma')=\{i_1, \ldots, i_k\}$ with $i_1<\cdots<i_k$.
Roughly speaking, we need to order the set of vertices
$V(\Gamma')$ in such a way that if $i\arr{} j$ in
$\Gamma'$ then $i<j$. Let us consider two cases.
\begin{enumerate}
    \item[1)] If $\{1, n\}\nsubseteq V(\Gamma')$ then it is clear that if $i\arr{}j$ in $\Gamma'$ then $i<j$.
    \item[2)] If $\{1, n\}\subseteq V(\Gamma')$ then there exist $t+1\leqslant k$ such that $1,\ldots, t\in V(\Gamma')$ but $t+1\notin V(\Gamma')$. Consider a graph $\Theta$ isomorphic to $\Gamma'$ with vertices $\{1,\ldots, k\}$ such that $l\arr{}m$ if and only if $i_{t+l}\arr{}i_{t+m}$ in $\Gamma'$, where for $t+j>k$ we assume that $i_{t+j}=i_{t+j-k}$. Denote by $\{a_1,\ldots, a_k\}$ the set of generators of the Hecke--Kiselman monoid $\HK_{\Theta}$ associated to this graph $\Theta$. Then the homomorphism
    $\HK_{\Theta}\longrightarrow \HK_{\Gamma'}$ defined on the generators
    by $a_j\mapsto x_{i_{t+j}}$ (modulo $k$) for $j=1,\ldots, k$ is an isomorphism.
    Moreover, it is clear that in the graph $\Theta$ if $l\arr{}m$
    then $l<m$. The form of idempotents in $\HK_{\Gamma'}$ in this
    case follows, via Theorem~\ref{id}, from the described isomorphism
    between $\HK_{\Theta}$ and $\HK_{\Gamma'}$.
\end{enumerate}
Therefore, we have proved the following
\begin{corollary}\label{idempform}
    Idempotents in $C_n$ are exactly the elements that can be written
    in the form $e_X$ for some $X\subseteq \{1,\ldots, n\}$ such that
    $X\neq\{1, \ldots, n\}$, where $e_X$ is defined as follows.
    \begin{enumerate}
        \item[1)]If $X=\emptyset$ then $e_X=1$.
        \item[2)] If $\{1, n\}\nsubseteq X\neq \emptyset$, then $e_X=x_{i_1}x_{i_2}\cdots x_{i_j}$ where $X=\{i_1,\ldots, i_j\}$ with $i_1<i_2<\cdots <i_j$.
        \item[3)] If $\{1, n\}\subseteq X$, let $k\in\{1, \ldots, n-1\}$ be such that $\{1, 2,\ldots,k\}\subseteq X$ and $k+1\notin X$. Then $e_X=x_{i_1}\cdots x_{i_s}x_1\cdots x_k$ where $i_s=n$ and $X=\{1, \ldots,k, i_1,\ldots, i_s\}$ with $k+1<i_1<\cdots <i_s=n$.
    \end{enumerate}
\end{corollary}

Recall that in $C_n$ there exists an ideal chain
$$\emptyset = I_{n-2}\subseteq I_{n-3}\subseteq\cdots\subseteq I_{0}\subseteq
I_{-1}$$ which is crucial in the description of the structure of
the algebra $K[C_n]$, see Proposition~\ref{struktura}.
\begin{proposition}For every idempotent $e_X$ in $C_n$ such that $|X|\geqslant 2$ we have
    $e_X\in I_{|X|-2}\setminus I_{|X|-1}$. Moreover $x_i\notin I_{-1}$
    for $i\in\{1,\ldots, n\}$.
\end{proposition}
\begin{proof}
    To prove the first statement, we will use the map
    $f:C_n\longrightarrow Map(\mathbb{Z}^n, \mathbb{Z}^n)$ considered
    in Section~3 of \cite{OW} given on generators $x_{i}$ of $C_{n}$
    as follows.
    \begin{equation*}
    f(x_i)(m_1, \ldots, m_n)= \begin{cases} (m_1, \ldots m_{i-1},
    m_{i+1}, m_{i+1},\ldots, m_{n})\textrm{ for }i\neq n\\(m_1,
    \ldots, m_{n-1}, m_1+1)\textrm{ for }i=n .\end{cases}
    \end{equation*}
    For every $i=-1, \ldots, n-2$ consider the following ideals
    \begin{equation*}
    Q_i=\{w\in C_n: |\supp (f(w))|\leqslant n-i-2 \}.
    \end{equation*}
    Here we adopt notation from Section~3 of \cite{OW}.
    Then, Lemma 3.7 in \cite{OW} states that for every $i=0,\ldots,
    n-2$ we have $Q_i\subseteq I_i$. Therefore, to show that $e_{X}\in
    I_{|X|-2}$ it is enough to check that $e_{X}\in Q_{|X|-2}$.  This
    follows from the following technical lemma.
    \begin{lemma}\label{nosnik-idempotent}
        Assume that $w=x_{i_1}\ldots x_{i_k}\in C_n$ is such that for
        every $j,l\in\{1,\ldots, k\}$ if $j<l$ then $0<i_l-i_j<n-1$ or
        $i_j-i_l\geqslant 2$. Then  $\supp (f(w))=\{1,\ldots,
        n\}\setminus\{i_1, \ldots, i_{k}\}$.
    \end{lemma}
    \begin{proof}
        We proceed by induction on the length of $w$. If $|w|=1$ then
        there is nothing to prove. Assume now that the assertion holds for
        some $k-1$, where $k\geqslant 2$. Consider the word
        $w=x_{i_1}\ldots x_{i_k}\in C_n$ satisfying the condition from the
        lemma. Then $f(w)(m_1,\ldots, m_n)=f(x_{i_1})(f(x_{i_2}\ldots
        x_{i_k})(m_1,\ldots, m_n))$. Applying the inductive hypothesis to
        the word $w'=x_{i_2}\ldots x_{i_k}$ we obtain that
        $\supp(f(x_{i_2}\ldots x_{i_k}))=\{1,\ldots,
        n\}\setminus\{i_2,\ldots, i_k\}$. Moreover, as $i_1\neq i_l$ for
        $l\geqslant 2$, we know that $f(x_{i_2}\ldots x_{i_k})(m_1,\ldots,
        m_n)$ has $m_{i_1}$ on the $i_1$ coordinate. It is enough to check
        that $m_{i_1}$ does not occur on any other coordinate of
        $f(x_{i_2}\ldots x_{i_k})(m_1,\ldots, m_n)$. Indeed, every $f(x_j)$ changes only $j$ coordinate.
        Therefore,  $f(w')(m_1,\ldots, m_n)$ could have $m_{i_1}$ on the coordinate different than $i_1$,
        only if $m_{i_1}$ had been rewritten firstly on the $i_1-1$ coordinate in the case $i_1\neq 1$
        or on the $n$ coordinate if $i_1=1$. Thus we get that either $x_{i_1-1}\in\{
            x_{i_2},\ldots,x_{i_k}\}$ or $i_1=1$ and $x_{n}\in\{
            x_{i_2},\ldots,x_{i_k}\}$, which is not true as
            $x_{i_l}-x_{i_1}\in\{ -n+1,\ldots, -2\}\cup\{1,\ldots, n-2\}$ for
            every $l>1$. It follows that $\supp(f(w))=\supp (f(w'))\setminus
            \{i_1\}=\{1,\ldots, n\}\setminus\{i_1,\ldots, i_k\}$.
    \end{proof}

    As every idempotent $e_X$ described in Corollary~\ref{idempform}
    satisfies conditions of Lemma~\ref{nosnik-idempotent}, it follows
    that $|\supp (f(e_X))|=n-|X|$,
    and thus $e_X\in I_{|X|-2}$.\\
    Let us now check that $e_{X}\notin I_{|X|-1}$. We have an
    automorphism $\sigma$ of $C_n$ such that $\sigma(x_i)=x_{i+1}$ for
    every $i=1,\ldots, n$ (where we agree that $x_{n+1}=x_1$). By
    Lemma 3.17 in \cite{OW},  it has the property that
    $\sigma(I_i)=I_i$ for $i\in\{0,\ldots, n-3\}$. Applying the
    automorphism $\sigma$ a few times if necessary, we may assume that
    the  idempotent $e_X$ is of the form $x_{i_1}\cdots x_{i_j}$ with
    $i_1<\cdots<i_j$ and $i_j=n-1$. Indeed, if $e_X=x_{i_1}\ldots
    x_{i_j}$ is such that $i_1<\cdots<i_j$ and $(i_1, i_j)\neq (1,
    n)$, then $\sigma^{n-1-i_j}(e_{X})=x_{n-(i_j+1)+i_1}\cdots
    x_{n-(i_j+1)+i_j}$ is of the required form as $1\leqslant
    n-(i_j+1)+i_k<n$ for $k=1,\ldots, j$. Moreover, if
    $e_{X}=x_{i_1}\cdots x_{i_s}x_n x_1\cdots x_k$ with
    $k+1<i_1<\cdots <i_s$ then $\sigma^{n-1-k}(e_X)=x_{i_1-k-1}\cdots
    x_{i_s-k-1}x_{n-k-1}\cdots x_{n-1}$, as
    $n+1\leqslant i_j+n-k-1<2n$ and thus $\sigma^{n-1-k}(e_X)$ is of the required form in this case, as well. \\
    So, let $e_X$ be of the form $x_{i_1}\cdots x_{i_j}$ with
    $i_1<\cdots<i_j$ and $i_j=n-1$.
    Then consider the element
    $$w=(x_{i_{1}-1}x_{i_{1}-2}\ldots x_{1})(x_{i_{2}-1}x_{i_{2}-2}\ldots x_{2})\cdots(x_{i_{j}-1}x_{i_{j}-2}\ldots x_{j}).$$
    Note that if $i_k=k$ for some $1\leqslant k\leqslant j$ then $i_1=1, \ldots, i_k=k$.
    If $i_k=k$ then by $(x_{i_{k}-1}x_{i_{k}-2}\ldots x_{k})$ we understand trivial word $1$.\\
    Since $x_{i_{m+1}}\cdots x_{i_{j}}$ commutes with
    $x_{i_{m}-1}\cdots x_{m}$ for $m=1,\ldots, j-1$, the element
    $e_{X}w$ can be written in $C_n$ in the form
    $$(x_{i_1}x_{i_1-1}\cdots x_{1})(x_{i_{2}}x_{i_{2}-1}\cdots x_{2})(x_{i_{3}}x_{i_{3}-1}\cdots x_{3})\cdots(x_{i_{j}}x_{i_{j}-1}\cdots x_{j}).$$
    From the description of the sets $A_i$ in Theorem 2.1 in \cite{OW}
    it follows that $e_Xw\in A_{j-1}$. Theorem 3.3 gives that there
    exist $u,v\in C_n$ such that $ue_{X}v=(x_nq_{j-1})^N$ for some
    $N\geqslant 1$. By the definition of the ideal $I_{j-1}$ this
    implies that $e_X\notin I_{j-1}=I_{|X|-1}$. The assertion follows.\\
    The second part of the proposition is clear from the definition
    of $I_{-1}$.

\end{proof}

\section{Prime ideals in $K[C_n]$}\label{primespectrum}

We are ready to describe the prime spectrum of the algebra
$K[C_n]$, for any $n\geqslant 3$.

For every $i=0,\ldots, n-2$, let us choose an ideal $J_i$ that is
maximal among all ideals in $K[C_n]$ such that $I_i\subseteq J_i$
and $J_i\cap K[x_nq_i]=0$.

\begin{theorem}\label{chain} Every maximal chain of prime ideals in the
    Hecke--Kiselman algebra $K[C_n]$ is of the form $J_i\subsetneq P$
    for $i=0,\ldots, n-2$, where $J_i$ are ideals described above
    and $P$ are maximal ideals in $K[C_n]$. Moreover, for every $i$ the ideal $J_i$ is uniquely
defined. Left (right) primitive ideals in $K[C_n]$ are precisely
the ideals~$P$.
\end{theorem}

\begin{proof}  We first show that ideals $J_i$ are the
only minimal prime ideals in $K[C_n]$.  From the definition it is
easy to check that every $J_i$ is prime. Moreover, as $J_i\cap
K[x_nq_i]=0$, the algebra $K[x_nq_i]=K[x]$ embeds into
$K[C_n]\big/ J_i$, and therefore $\GK K[C_n]\big/ J_i\geqslant 1$.
It was proved in \cite{mecel_okninski1} that $\GK K[C_n]=1$ and
thus $\GK K[C_n]\big/ J_i = 1$. We know also that $K[C_n]$ is a
noetherian PI algebra, see \cite{mecel_okninski1} and \cite{OW}.
From Corollary 10.16 in \cite{krause} it follows that $1=\GK
K[C_n]=\GK K[C_n]\big/ J_i=\textrm{clK}\dim K[C_n]\big/ J_i$. Thus
$J_i$ are minimal prime ideals in $K[C_n]$. Moreover, the intersection $J=J_0\cap J_1\cap\ldots\cap J_{n-2}$ is an ideal in
$K[C_n]$ and we claim that $J\cap K[M]=0$. 
First, recall that the contracted algebra $K_{0}[M_i]$ is a prime algebra for any $i=0,\ldots, n-2$, see Theorem~5.8 in \cite{OW}. Since $(x_nq_i)K_{0}[M_i](x_nq_i)\subseteq K[x_nq_i]$, then also $x_nq_i(K_{0}[M_i]\cap J_i)x_nq_i\subseteq K[x_nq_i]\cap J_i$ and it follows that $ K[x_nq_i]\cap J_i=0$ implies that $ K[M_i]\cap J_i=0$. Suppose that
$\alpha\in J\cap K[M]$ is a non-zero element. Then we can write
$\alpha=\alpha_0+\ldots+ \alpha_{n-2}$, where $\alpha_k\in K[M_k]$
for every $k$. Take minimal $i$ such that $\alpha_i\neq 0$. Then
$i=n-2$ would imply $\alpha=\alpha_{n-2}\in K[M_{n-2}]\cap J=0$,
which leads to a contradiction. Therefore $i<n-2$. Consider $0\neq
\alpha_i=\alpha-(\alpha_{i+1}+\ldots+\alpha_{n-2})$. On the one
hand, we have that $\alpha_i\in K[M_i]$. On the other hand
$\alpha_{i+1}+\ldots+\alpha_{n-2}\in K[M_{i+1}]+\ldots+
K[M_{n-2}]\subseteq K[I_i]\subseteq J_i$ and thus
$\alpha-(\alpha_{i+1}+\ldots+\alpha_{n-2})\in J_i$. Therefore
$0\neq\alpha_i\in K[M_i]\cap J_i=0$. The obtained contradiction
shows that $J\cap K[M]=0$, as claimed. Since $C_n\setminus M$ is
finite, $J$ is finite dimensional and thus $J=0$ follows from
Lemma 3.3 in \cite{OW}. This implies that $J_i$ are the only
minimal prime ideals in $K[C_n]$ and for every prime ideal $Q$ in
$K[C_n]$ we have $J_i\subseteq Q$ for some $i\in\{0,\ldots,
n-2\}$. So, actually, for every $i$ the ideal $J_i$ is uniquely
defined.

Moreover, $\textrm{clK}\dim K[C_n]\big/ J_i=\GK K[C_n]\big/
J_i=1$, so $J_i$ is not a maximal ideal of $K[C_n]$ (as a simple
PI-algebra is finite dimensional, by the theorem of Kaplansky; see
Theorem 1.5.16 in \cite{Rowen}). Hence, for every $i\in\{0,\ldots,
n-2\}$ there exists a prime ideal $P$ such that $J_i\subsetneq P$.
Then $\textrm{clK}\dim K[C_n]\big/ P=0$ and $\textrm{clK}\dim
K[C_n]\big/ P=\GK K[C_n]\big/ P$, so $K[C_n]\big/ P$ is a finite
dimensional prime algebra. It follows that $P$ is a maximal ideal
in $K[C_n]$. Clearly for every $i$ there may be many maximal
ideals $P$ containing $J_i$.

 In particular (left) right primitive ideals of $K[C_n]$ are
exactly such ideals $P$.
\end{proof}

\section{Irreducible representations of $K[C_n]$}\label{repC_n}
Our aim is to investigate the irreducible representations of the
Hecke--Kiselman algebra $K[C_n]$. Let us remind that the field $K$ is algebraically closed.

Our approach is based on Proposition~\ref{struktura}. The general
idea is motivated by the representation theory of finite
semigroups. Namely, if the semigroup $S$ is finite, then the
irreducible representations of $S$ can be obtained in terms of
irreducible representations of maximal subgroups of $S$, see
Chapter 5 in \cite{clifford}, or Chapter 5 in \cite{steinberg} for
another approach. The latter representations are related to
$0$-simple principal factors of the semigroup $S$. Moreover, it is
known that the Hecke--Kisleman monoid $C_n$ is
$\mathcal{J}$--trivial, see Theorem~4.5.3 in \cite{denton2}.
Representation theory of finite $\mathcal{J}$--trivial monoids
(more generally finite $\mathcal{R}$--trivial monoids) can be
easily described, see Corollary~5.7 in \cite{steinberg} or
\cite{denton}. In particular, in the $\mathcal{J}$--trivial case
representations can be parametrized by idempotents. We construct
two types of representations of $K[C_n]$: those coming from the
representations of $K[M_i]$, and those related to the idempotents
in $C_n$. It is known that sandwich matrices $P_i$ are invertible
as matrices in $M_{n_i}(K(s_i))$, see \cite{semiprime}, and thus $K_0[M_i]$ are almost
Munn algebras. Simple modules over such algebras can be
described, see Chapter 5 in \cite{semalg}. In Section~\ref{repM} we will extend this result to our setting.

First, we are going to prove that every maximal ideal $J$ in
$K_0[M_i]$ extends to maximal ideal in the algebra $K[C_n]$.

Abusing slightly the notation, $K_{0}[M_i]$ will be sometimes
identified with the subspace $K[M_i]$ of the algebra $K[C_n]$.
\begin{proposition}\label{simple Mi}
Assume that $I$ is a maximal ideal in $K_0[M_i]$. Then there
exists a unique maximal ideal $\tilde{I}$ in $K[C_n]$ such
that $\tilde{I}\cap K_{0}[M_i] =I$ and $K_0[M_i]/I\simeq
K[C_n]/\tilde{I}$.
\end{proposition}
\begin{proof}
We treat the semigroup $M_i$ as a subsemigroup in $I_{i-1}/I_i$.
Then we know that $M_i$ is a two-sided ideal in $C_n/I_{i}$,
see Proposition~\ref{struktura}. Let $I$ be the ideal in
$K_0[M_i]$ such that the algebra $K_0[M_i]/I$ is simple. Then, as
$K_0[M_i]$ is an ideal in $K_0[C_n/I_{i}]$, we have that
$I\subseteq K[C_n] I K[C_n]\subseteq K[M_i]$ modulo $K[I_i]$. We
assume that $K_0[M_i]/I$ is simple and therefore (modulo $K[I_i]$)
either $I= K[C_n] I K[C_n]$ or $K[C_n] I K[C_n]= K[M_i]$. On the
other hand
\begin{eqnarray*} \lefteqn{ (K[C_n](I+K[I_i])K[C_n])^3  \subseteq }\\
& &(K[C_n] (I+K[I_i])
K[C_n])(I+K[I_i])(K[C_n] (I+K[I_i]) K[C_n])\subseteq \\
& &K[M_i](I+K[I_i])K[M_i] + K[I_{i}]\subseteq I+K[I_i] .
\end{eqnarray*}
It follows that $I= K[C_n] I K[C_n]$ modulo $K[I_i]$. Therefore
$\overline{I}=K[I_i]+I$ is an ideal in $K[C_n]$. Moreover from
$\overline{I}\cap K_0[M_i]=I$ (as $M_i\subseteq I_{i-1}\setminus
I_i$) we get that $K_0[M_i]/I\triangleleft K[C_n]/\overline{I}$.
$K_0[M_i]/I$ is a simple PI-algebra and hence it is finite
dimensional by Kaplansky theorem.

 Thus $K_0[M_i]/I$ is an algebra with an identity and
an ideal of $K[C_n]/\overline{I}$, so that
$K_0[M_i]/I=(K[C_n]/\overline{I})\cdot f$ for a central
idempotent~$f$. If we consider the natural epimorphism $\varphi:
K[C_n]\twoheadrightarrow (K[C_n]/\overline{I})\cdot f$ then
$\tilde{I}=\ker\varphi$ is an ideal in $K[C_n]$ such that
$K[C_n]/\tilde{I}\simeq K_0[M_i]/I$. The uniqueness of the
ideal $\tilde{I}$ is a direct consequence of the construction
because $\bar{I}$ is the ideal of $K[C_n]$ generated by $I$.
\end{proof}

Note that the maximal ideal in $K[C_n]$ described in
Proposition~\ref{simple Mi} is one of the ideals $P$ associated to
the minimal prime ideal $J_i$ in Theorem \ref{chain} and we have
$K_0[M_i]/I\simeq K[C_n]/P\simeq M_j(K)$ for some $j\geqslant 1$
(we assume that $K$ is algebraically closed). Then $P$ is the
annihilator of a simple module and the corresponding simple module
is $K^j$.

Now, consider any maximal ideal $P\triangleleft K[C_n]$. As we have
explained earlier, $J_i\subseteq P$ for some $i\in\{0,\ldots,
n-2\}$ and thus also $I_j\subseteq P$ for some $j\in\{0,\ldots,
n-2\}$. Take the minimal $i\geqslant 0$ such that $I_i\subseteq
P$. Then either $K[M_i]\nsubseteq P$ or $K[M_i]\subseteq P$. In
the first case $P\cap K[M_i]\neq 0$, as otherwise $P=J_i$, which
contradicts Theorem \ref{chain}. Then $0\neq P\cap K_0[M_i]$ is
ideal in $K_0[M_i]$ such that $K_0[M_i]\big/(P\cap
K_0[M_i])\triangleleft K[C_n]/P\simeq M_j(K)$ for some $j\geqslant
1$. Therefore we get that $K_0[M_i]\big/(P\cap K_0[M_i])=K[C_n]/P$. In
particular, in this case the maximal ideal $P$ comes from a maximal
ideal in $K_0[M_i]$, in the way described in Proposition~\ref{simple
Mi}.

Now, let us consider the second case, namely $K[M_i]\subseteq P$.

If $i=0$ and $I_{-1}\subseteq P$, then $C_n\big/ I_{-1}$ is a
finite monoid and from Section~\ref{idempotents} it follows that there are
exactly $n+1$ idempotents $e=e^2\in C_n\setminus I_{-1}$ (generators $x_1, \ldots, x_n$ and $1$).
As $C_n$ is a $\mathcal{J}$-trivial monoid, also $C_n/I_{-1}$ is $\mathcal{J}$-trivial.

By Corollary 5.7 in \cite{steinberg}, all irreducible
representations of the monoid $C_n/I_{-1}$ can be parametrized by
idempotents $e\in C_n\setminus I_{-1}$. More precisely, let
$e^2=e$ be such an idempotent and consider $\overline{I}(e)=\{m\in
C_n/I_{-1}: e\notin (C_n/I_{-1}) m (C_n/I_{-1})\}$, that is the
ideal of non-generators of the principal ideal $(C_n/I_{-1}) e
(C_n/I_{-1})$. Then there exists a one-dimensional simple
$K[C_n/I_{-1}]$-module $S_e$ such that the corresponding
representation $\varphi_e: K[C_n/I_{-1}]\longrightarrow K$ is
given for $m\in C_n/I_{-1}$ by
$$\varphi_e(m)=\begin{cases}  0 \quad \textrm{ if }m\in \overline{I}(e)
\\ 1 \quad \textrm{ otherwise }\end{cases}.$$
This induces the irreducible representation of $K[C_n]$ given by
$K[C_n]\twoheadrightarrow K[C_n/I_{-1}]\arr{\varphi_e}{K}$ and
thus leads to a maximal ideal in $K[C_n]$ associated to $e$.
What is more, for different idempotents $e=e^2\in C_n\setminus
I_{-1}$ we get different maximal ideals and in this case all
possible maximal ideals are of the above form.

Assume now that $K[M_i]\subseteq P$ and $i>0$, that is
$I_{i-1}\nsubseteq P$. Consider the finite semigroup
$I_{i-1}\big/(I_i\cup M_i)$. In this case every simple $K[C_n]$-module with annihilator $P$ is also a
simple $K[I_{i-1}\big/(I_i\cup M_i)]$-module. The latter modules can be described,
see Chapter~5 in \cite{steinberg}. Indeed, from the characterization of idempotents
in Section~\ref{idempotents} it is clear that there exists an idempotent $e\in I_{i-1}\big/(I_i\cup M_i)$
and among
all such idempotents we can choose a minimal one with respect to the
$\mathcal{J}$-relation in $I_{i-1}\setminus (I_i\cup M_i)$ such that for its
$\mathcal{J}$-class $\mathcal{J}_e$ we have $\mathcal{J}_e\nsubseteq P$. That is $e$ is an apex of a
simple $K[I_{i-1}\big/(I_i\cup M_i)]$-module $W$ with $K[\overline{I}(e)]=Ann(W)$, where
$\overline{I}(e)=\{w\in I_{i-1}\big/(I_i\cup M_i):
e\notin I_{i-1}\big/(I_i\cup M_i)wI_{i-1}\big/(I_i\cup M_i)\}$
(we know that every simple $K[I_{i-1}\big/(I_i\cup M_i)]$-module has an apex).

 Next, consider the ideal $N_e=I_i\cup M_i\cup I(e)$
in $C_n$, where $I(e)=\{w\in C_n: e\notin C_nwC_n\}$. It follows
that $\{\theta, e\}\triangleleft C_n/N_e$ ($C_n$ is
$\mathcal{J}$--trivial, see also Corollary 2.7 in
\cite{steinberg}) and thus in particular $K[C_n]\big/K[N_e]\simeq
Ke\oplus (K[C_n]\big/K[N_e])(1-e)$. It follows that in this case
the irreducible representation $\varphi$ of $K[C_n]$ is
one-dimensional and given by $ K[C_n]\twoheadrightarrow
K[C_n/N_e]\arr{ \varphi_e} K$, where for any $m\in C_n\setminus
N_e$
$$\varphi_e(m)=\begin{cases}  1 \textrm{ if }em=e \\ 0 \textrm{ if }em\in N_e\end{cases}.$$
As in the previous case, for every choice of an idempotent
$e=e^2\in I_{i-1}\setminus (M_i\cup I_i)$ we get a different
maximal ideal $P_e$ and in such a way we get all maximal ideals
in this case.

Thus we have proved the following theorem.
\begin{theorem}   \label{irred main}
Let $\varphi :K[C_n]\longrightarrow M_{j}(K)$ be an irreducible
representation of the Hecke--Kiselman algebra $K[C_n]$ over an
algebraically closed field $K$. If $\varphi(K[I_{n-3}])\neq0$ set
$i=n-2$. Otherwise take the minimal $i\in\{-1,\ldots, n-3\}$ such
that $\varphi(K[I_i])=0$.
\begin{enumerate}
    \item[1)] If $i\geqslant 0$ then
    \begin{enumerate}
        \item[a)] either $\varphi(K[M_i])\neq 0$ and the representation $\varphi$ is induced by a
         representation of $K_0[M_i]$;
        \item[b)] or $\varphi(K[M_i])=0$ and the representation $\varphi$ is one-dimensional and
        induced by an idempotent $e\in I_{i-1}\setminus I_{i}$. In particular $\varphi(e)=1$.
    \end{enumerate}
    \item[2)] If $i=-1$ then the representation $\varphi$ is one-dimensional and induced by an
    idempotent $e\in C_n\setminus I_{-1}$. In particular $\varphi(e)=1$.
\end{enumerate}
\end{theorem}

\section{Maximal ideals in $K_0[M_i]$}\label{repM}
In view of Theorem~\ref{irred main}, in order to complete our
discussion of irreducible representations of $K[C_n]$, the final
step is to characterize irreducible representations of $K_0[M_i]$.
Recall that the field $K$ is algebraically closed. We start with a
construction of certain family of such representations using the
representation theory of Munn algebras, see Chapter~5 in
\cite{semalg}.

\begin{proposition}\label{M_i1}
Consider the semigroup of matrix type $M_i=\mathcal{M}^0(S_i, A_i,
B_i; P_i)$ occurring in Proposition~\ref{struktura}, where $S_i$
is the infinite cyclic semigroup generated by $s_i$. Then for
every $\lambda\in K^*$ there exists a unique irreducible
representation $\psi_{\lambda}:K_0[M_i]\longrightarrow M_r(K)$,
for some $ 1\leqslant r\leqslant |A_i|$, induced by the
irreducible representation of the cyclic group generated by $s_i$.
\end{proposition}
\begin{proof}
For every $i\in\{0,\ldots, n-2\}$ we have that
$M_i=\mathcal{M}^0(S_i, A_i, B_i; P_i)$, where $S_{i}$ is the
cyclic semigroup generated by $s_i=x_nq_i$, $P_i$ is a matrix of
size $B_i\times A_i$ and with coefficients in $\langle
x_nq_i\rangle\cup\{\theta\}$. Moreover, we know that $|A_i|=|B_i|$
and the sandwich matrix $P_i$, being not a zero divisor in the
matrix ring $M_{n_i}(K[s_i])$, is invertible as a matrix in
$M_{n_i}(K(s_i))$, where $K(s_i)$ is the field of rational
functions in the indeterminate $s_i$. Therefore, every algebra
$K_0[M_i]$ embeds into the algebra of matrix type
$K_0[\mathcal{M}^0(\gr(s_i), A_i, B_i, P_i)]\simeq
\mathcal{M}(K[s_i, s_i^{-1}], A_i, B_i, P_i)$, where $\gr(s_i)$ is
the infinite cyclic group generated by $s_i$. For any $\lambda\in
K^{*}$ consider the homomorphism $\varphi_{\lambda}: K[s_i]\arr{}
K$ such that $s_i\mapsto \lambda$. Then it extends to the
homomorphism $\mathcal{M}(K[s_i, s_i^{-1}], A_i, B_i,
P_i)\longrightarrow \mathcal{M}(K, A_i, B_i, \overline{P}_i)$,
where $\overline{P}_i$ is the matrix
$\overline{P}_i=(\overline{p}_{b, a})$ with $\overline{p}_{b,
a}=\varphi(p_{ba})$ for every $a\in A_i$, $b\in B_i$. Then there
exists an epimorphism $\mathcal{M}(K, A_i, B_i,
\overline{P}_i)\longrightarrow M_{r}(K)$, where
$r=\textrm{rk}(\overline{P}_i)=\sup\{k: \overline{P}_i\textrm{ has
an invertible }k\times k\textrm{ submatrix}\}$ and the induced map
$\mathcal{M}(K[s_i, s_i^{-1}], A_i, B_i, P_i)\longrightarrow
\mathcal{M}(K, A_i, B_i, \overline{P}_i)\longrightarrow M_{r}(K)$
is also an irreducible representation, see Chapter~5 in
\cite{semalg}. It follows that the restriction of the above map
gives an irreducible representation $\psi_{\lambda}:
K_0[M_i]\longrightarrow M_{r}(K)$.
\end{proof}

Recall from the definitions of the sets $A_i$, $B_i$ and sandwich
matrices $P_{i}$ in the proof of Corollary 3.23 in \cite{OW} that
there exist $a\in A_i$, $b\in B_i$ such that $p_{ba}=1$.
Identifying finite sets $A_i$ and $B_i$ with $\{1,\ldots,
{n\choose i+1}\}$ we can assume that $p_{1,1}=1$. Then
$M_{1,1}\subseteq M_i$, defined as $\{(s_i^k; 1, 1)\in M_i:
k\geqslant 1\}$, is isomorphic to the infinite cyclic semigroup.

The irreducible representations of $M_i$ described in
Proposition~\ref{M_i1} come from representations of the completely
0--simple closure of $M_i$, namely
$cl(M_i)=\mathcal{M}^0(\gr{(s_i)}, A_i, B_i, P_i)$, where
$\gr{(s_i)}$ stands for the cyclic group generated by $s_i$. In
particular, the image of $(s_i, 1, 1)$ is equal to $\lambda
e$ for an idempotent matrix $e$ of rank $1$ and some $\lambda \in
K^{*}$. Now we show that every irreducible representation of
$M_i$ is of such a form and extends to a representation of
$\mathcal{M}^0(\gr(s_i), A_i, B_i, P_i)$.

\begin{proposition}\label{repM_i}
Every irreducible representation $\varphi: K_0[M_i]\longrightarrow
M_{r}(K)$ of $K_0[M_i]$ is such that $\varphi(s_i; 1, 1)=\lambda
e$, where $\lambda\in K^{*}$ and $e$ is an idempotent of rank $1$.
Such a representation $\varphi$ can be uniquely extended to an
irreducible representation of $K_0[\mathcal{M}^0(\gr(s_i), A_i,
B_i, P_i)]$.
\end{proposition}

Consider any irreducible representation $\varphi:
K_0[M_i]\longrightarrow M_{r}(K)$, where $M_i=\mathcal{M}^0(S_i,
A_i, B_i; P_i)$. In particular, as $K$ is algebraically closed,
the representation is onto. The first step of the proof is to
investigate the image of $M_{1,1}$ under $\varphi$. Let us notice that
$\varphi(M_{1,1})$ is nonzero. Indeed, if we had in particular
$\varphi(s_i; 1, 1)=0$, then for any $(s_i^k; a, b)\in M_i$ we
would have $\varphi((s_i^k; a, b)^{N_0})=\varphi(s_i; a,
1)\varphi(s_i^{N_{ba}}; 1, 1)\varphi(s_i; 1, b)=0$ for $N_0$ big
enough and an appropriately chosen exponent $N_{ba}$.
Therefore $\varphi(M_i)$ would be nil and thus nilpotent in the
monoid $(M_r(K), \cdot)$, which leads to a contradiction. A
similar argument shows that $\varphi(s_i, 1, 1)$ cannot be
nilpotent.

\begin{lemma}\label{domain}
The image $\varphi(K[M_{1,1}])$ is a commutative integral domain.
Consequently, $\varphi(K[M_{1,1}])\simeq K$.
\end{lemma}
\begin{proof}
Note that in the sandwich matrix $P_i$ we have $p_{1, 1}=1$, that is $K[M_{1,1}]$
is isomorphic to $K[s_i]s_i$ and in particular $\varphi(K[M_{1,1}])$ is commutative.\\
Suppose now that $\varphi(K[M_{1,1}])$ is not an integral domain, that is $uv=0$
for some nonzero $u, v\in \varphi(K[M_{1,1}])$ with $u=\varphi(\alpha)$,
$v=\varphi(\beta)$, where $\alpha=(f; 1, 1)$, $\beta=(g; 1, 1)$.
Define ideals in $K_0[M_i]$ as follows.
$$I=\{y=(y_{ab})\in K_{0}[M_i]: y_{ab}\in K[s_i]f \textrm{ for all } a\in A_i, b\in B_i\}$$
$$J=\{y=(y_{ab})\in K_{0}[M_i]: y_{ab}\in K[s_i]g \textrm{ for all } a\in A_i, b\in B_i \}$$
Then
$IJ\subseteq\{y=(y_{ab})\in K_{0}[M_i]: y_{ab}\in K[s_i]fg \textrm{ for
all } a\in A_i, b\in B_i\}$ and as $\varphi(I)\neq 0$,
$\varphi(J)\neq 0$, it follows that
$\varphi(I)=\varphi(J)=M_r(K)$. Therefore,
$$\varphi((s_i; 1, 1)IJ(s_i; 1, 1))=\varphi(s_i; 1,
1)M_{r}(K)\varphi(s_i; 1, 1)\neq 0.$$ On the other hand
$\varphi((s_i; 1, 1)IJ(s_i; 1, 1))\subseteq\varphi((K[s_i]fg; 1,
1))=0$, as $\varphi((f; 1, 1)(g; 1, 1))=uv=0$, which leads to a
contradiction. Consequently, $\varphi(K[M_{1,1}])$ is a
finite dimensional domain and, as $K$ is algebraically closed, it
follows that $\varphi(K[M_{1,1}])\simeq K$.
\end{proof}

Now we are ready to prove Proposition~\ref{repM_i}.
\begin{proof}
As explained earlier, if we denote $s=(s_i; 1, 1)$, then
$\varphi(s)$ is nonzero and not nilpotent. Therefore for $k$
big enough $\varphi(s)^k$ lies in a maximal subgroup of the monoid
$(M_n(K),\cdot)$. Then $\varphi:
s^{k}K_0[M_i]s^{k}\longrightarrow\varphi(s)^kM_{r}(K)\varphi(s)^k=eM_{r}(K)e$
for some idempotent $e$. Moreover, by Lemma~\ref{domain},
$s^{k}K_{0}[M_i]s^{k}$ is contained in a commutative ring
$K[M_{1,1}]$, thus it follows that $\textrm{rk}(e)=1$. Now, we
have that $ \varphi(s)=g\begin{pmatrix} a_0 & 0 \\ 0 & a_1
\end{pmatrix} g^{-1}
$ for matrix $a_0\in M_{p}(K)$, such that  $\textrm{rk}
((a_0)^{2}) =\textrm{rk}(a_{0})$, a nilpotent matrix $a_1\in
M_{r-p}(K)$, where
$g\in Gl_{r}(K)$ and $\varphi(s^k)=g\begin{pmatrix} a_0^{k} & 0 \\
0 & 0
\end{pmatrix} g^{-1}$. As $\varphi(s^k)\in eM_r(K)e$ for an idempotent $e$ of rank $1$,
it follows that $a_0\in K^{*}$. Moreover, if $a_1\in M_{r-1}(K)$ is nonzero,
then $g\begin{pmatrix} a_0^k & 0 \\ 0 & 0
\end{pmatrix} g^{-1}$,
$g\begin{pmatrix} a_0 & 0 \\ 0 & a_1
\end{pmatrix} g^{-1}\in \varphi(K[M_{1,1}])$, so $\dim_{K}\varphi(K[M_{1,1}])\geqslant 2$,
which contradicts the assertion of Lemma~\ref{domain}. Thus in particular
$\varphi(s_i; 1, 1)=\lambda e$ for $\lambda\in K^{*}$ and $e=e^2$ with $\textrm{rk}(e)=1$.\\
Now consider the irreducible representation $\varphi:
K_0[M_i]\longrightarrow M_r(K)$ such that $\varphi(s_i; 1,
1)=\lambda e$. As $M_i=\{(s_i^m; a, b) : m\geqslant 1, a\in A_i,
b\in B_i\}$ and the sandwich matrix $P_i$ has $p_{1,1}=1$, we have
$\varphi(s_i^m; a, b)=\varphi(s_i; a, 1)\varphi(s_i^{m-2}; 1,
1)\varphi(s_i; 1, b)= \varphi(s_i; a, 1)\lambda^{m-2}e\varphi(s_i;
1, b)$ for $m\geqslant 3$. Let us define the extension $\overline{\varphi}:
K_0[\mathcal{M}^0(\gr(s_i), A_i, B_i; P_i)]\longrightarrow
M_{r}(K)$ for any $(s_i^{p}; a, b)$, where $p\in\mathbb{Z}$, $a\in
A_i$, $b\in B_i$ by
$$\overline{\varphi}(s_i^{p}; a, b) = \varphi(s_i; a, 1)\lambda^{p-2}e\varphi(s_i; 1, b).$$
To verify that $\overline{\varphi}$ is a homomorphism take any
$(s_i^n; a, b)$, $(s_i^m; a', b')\in K_{0}[M_i]$. Then either
$(s_i^n; a, b) (s_i^m; a', b')=(s_{i}^{n+m+n_{ba'}}; a, b')$ if
$p_{ba'}=s_i^{n_{ba'}}$ or $(s_i^n; a, b)(s_i^m; a', b')=0$ if
$p_{ba'}=\theta$. In the first case also $(s_i; 1, b)\cdot (s_i;
a', 1)=(s_i^{n_{ba'}+2}; 1, 1)$. We then get
\begin{align*}
\overline{\varphi}(s_i^n; a, b)\overline{\varphi}(s_i^m; a', b')&=
\varphi(s_i; a, 1)\lambda^{n-2}e\varphi((s_i; 1, b)\cdot (s_i; a', 1))\lambda^{m-2}e\varphi(s_i; 1, b')=
\\&= \varphi(s_i; a, 1)\lambda^{n-2}e\varphi((s_i^{n_{ba'}+2}; 1, 1)\lambda^{m-2}e\varphi(s_i; 1, b')=
\\&= \varphi(s_i; a, 1)\lambda^{n+m+n_{ba'}-2}e\varphi(s_i; 1, b')=\\&=
\overline{\varphi}(s_{i}^{n+m+n_{ba'}}; a, b').
 \end{align*}
Moreover $(s_i^n; a, b)(s_i^m; a', b')=0$ if and only in $(s_i; 1, b)(s_i; a', 1)=0$,
and thus in the second case $$\overline{\varphi}(s_i^n; a, b)\overline{\varphi}(s_i^m; a', b')=
\overline{\varphi}((s_i^n; a, b) (s_i^m; a', b'))=0.$$

Therefore $\overline{\varphi}$ is a homomorphism. Let us denote
$N_{i}=\{(s_i^k; a, b): k\geqslant 3, a\in A_i, b\in
B_i\}\subseteq M_i$. Then $K_0[N_i]$ is an ideal in $K_0[M_i]$ and
it is clear that
$\overline{\varphi}|_{K_{0}[N_i]}=\varphi|_{K_{0}[N_i]}$. Moreover
$\varphi(K_0[N_i])=M_{r}(K)$ and in particular $\varphi(f)=1$ for
some $f\in K_{0}[N_i]$. Then for any $g\in K_{0}[M_i]$ we have
$\varphi(g)=\varphi(g)\varphi(f)=\varphi(gf)$. It follows that
$\varphi$ is uniquely determined by $\varphi|_{K_0[N_i]}$.
Therefore $\overline{\varphi}|_{K_0[M_i]}=\varphi$, in other words
the irreducible representation $\varphi$ can be extended to the
representation $\overline{\varphi}$ of
$K_0[\mathcal{M}^0(\gr(s_i), A_i, B_i; P_i)]$. Uniqueness of the
extension follows from the classification of irreducible
representations of completely $0$-simple semigroups.
\end{proof}

\section{Irreducible representations of $K[C_3]$}
Let us illustrate  the results with the simplest case of
the Hecke--Kisleman monoid $C_3$ associated to the cycle of length $3$.
We start with the representations induced by the irreducible representations of
the semigroups of matrix type inside the Hecke--Kiselman algebra $K[C_3]$.
Recall from \cite{OW} that in the monoid $C_3$ we have two semigroups of matrix
type $M_0$ and $M_1$. More precisely, let $T$ be the cyclic semigroup generated by $t=x_3x_1x_2$.
Then $M_1=(C_3x_3x_1x_2C_3)^{0}$ is
a semigroup of matrix type $\mathcal{M}^0(T, A_1, B_1; P_1)$, where
$A_1=\{1, x_2, x_1x_2\}$, $B_1=\{1, x_3, x_3x_1\}$, with sandwich matrix
(with coefficients in $T^1$)
\begin{equation*}
P_1=\begin{pmatrix}
1 & 1 & 1\\1 & 1 & t\\ 1 & t & t
\end{pmatrix}.
\end{equation*}

Then if $\lambda \neq 0, 1$ we have that $\overline{P}_1$ is a
matrix of rank $3$ and therefore we get the family of
representations $\psi_{\lambda}: K_0[M_1]\longrightarrow
M_{3}(K)$. In this case the epimorphism $\mathcal{M}(K, A_1, B_1,
\overline{P}_1)\longrightarrow M_{3}(K)$ is given by $A\mapsto
A\circ \overline{P}_1$. Therefore, the representation
$\psi_{\lambda}$ is given for every $(t^k;a, b)\in M_1$ by

$$(t^k; a, b)\mapsto M_{(\lambda^k; a, b)}\circ \begin{pmatrix}
1 & 1 & 1\\1 & 1 & \lambda\\ 1 & \lambda & \lambda
\end{pmatrix}, $$
where $\circ$ is the standard matrix multiplication and $ M_{(\lambda^k; a, b)}\in M_{3}(K)$ is a matrix
with the only non-zero entry $(a, b)$ equal to $\lambda^k$.

For $\lambda=1$ the matrix $\overline{P}_i$ is of rank $1$ and
therefore we get the one-dimensional representation $\psi_{1}:
K_0[M_1]\longrightarrow K$.

Similarly, let $S$ be the cyclic semigroup generated by  $s=x_3x_2x_1$. Then the semigroup
$M_0=I_{-1}/I_0$ is a semigroup of matrix type $\mathcal{M}^0(S, A_0, B_0; P_0)$,
$A_0=\{1, x_1, x_2x_1\}$, $B_0=\{1, x_3, x_3x_2\}$, with sandwich matrix (with coefficients in
$S^1\cup\{\theta\}$)
\begin{equation*}
P_0=\begin{pmatrix}
1 & 1 & \theta\\1 & \theta & s\\ \theta & s & s
\end{pmatrix}.
\end{equation*}
Then, if $\lambda \neq 0, -1$ then we have that $\overline{P}_0$
is a matrix of rank $3$ and therefore we get the family of
representations $\psi_{\lambda}: K_0[M_0]\longrightarrow
M_{3}(K)$. The representation $\psi_{\lambda}$ is given for every
$(s^k;a, b)\in M_0$ by

$$(s^k; a, b)\mapsto M_{(\lambda^k; a, b)}\circ \begin{pmatrix}
1 & 1 & 0\\1 & 0 & \lambda\\ 0 & \lambda & \lambda
\end{pmatrix}, $$
where $\circ$ is the standard matrix multiplication and $ M_{(\lambda^k; a, b)}\in M_{3}(K)$
is a matrix with the only non-zero entry $(a, b)$ equal to $\lambda^k$.

Moreover, for $\lambda=-1$ the matrix $\overline{P}_0$ has rank
$2$ and therefore the corresponding representation is $\psi_{-1}:
K_0[M_0]\longrightarrow M_{2}(K)$.

Note also that for $\lambda=0$ and any $i$ the induced
homomorphism $K_0[M_i]\longrightarrow \mathcal{M}(K, A_i,
B_i;\overline{P_i})$ is the zero map.

From the results in Section~\ref{idempotents} it follows that
$\{1, x_1, x_2, x_3, x_1x_2, x_2x_3, x_3x_1\}$ is the set of
idempotents in $C_3$ and $1, x_1, x_2, x_3\in C_3\setminus
I_{-1}$, $x_1x_2, x_2x_3, x_3x_1\in I_0\setminus I_{1}$ (note that
$I_1=\emptyset$). Thus, we get (Theorem~\ref{irred main}) that
irreducible representations of $K[C_3]$ either come from the
representations of $K_0[M_0]$ or $K_0[M_1]$ described above or are
one-dimensional representations associated to one of the
idempotents in the monoid $C_3$.

\section{Representations of PI Hecke--Kiselman algebras}

In this section we describe irreducible representations of
arbitrary Hecke--Kiselman algebras satisfying a polynomial identity. Recall
from \cite{mecel_okninski1} that this condition can be
characterized by certain property of the corresponding graph.
Namely, $K[\HK_{\Theta}]$ is a PI-algebra if and only if $\Theta$
does not contain two  cyclic subgraphs (that is subgraphs which
are oriented cycles) connected by an oriented path.

The radical of the algebra $K[\HK_{\Theta}]$ in this case was described in Theorem~2 in
\cite{semiprime}. Assume that $\Theta'$ is the subgraph of
$\Theta$ obtained by deleting all arrows $x\rightarrow y$ that are
not contained in any cyclic subgraph of $\Theta$. Every connected
component of $\Theta'$ is either a singleton or an oriented cycle.
Then $K[\HK_{\Theta}]/{\mathcal J}(K[\HK_{\Theta}])\cong
K[\HK_{\Theta'}]$ and it is the tensor product of algebras
$K[\HK_{\Theta_{i}}]$ of the connected components
$\Theta_1,\ldots, \Theta_m$ of $\Theta'$, each being isomorphic to
$K\oplus K$ or to the algebra $K[C_{j}]$, for some $j\geqslant 3$.

\begin{theorem}\label{rep-pi}
Assume that $\Theta$ is a finite oriented graph such that
$K[\HK_{\Theta}]$ is a PI-algebra and $\Theta'$ is the subgraph of
$\Theta$ as described above, with the connected components
$\Theta_1,\ldots, \Theta_m$.

Then the maximal ideals of $K[\HK_{\Theta}]$
are in a bijection with maximal ideals of $K[\HK_{\Theta'}]$. The
latter maximal ideals are of the form
$$\sum_{i=1}^{m}K[\HK_{\Theta_{1}}]\otimes\cdots \otimes
K[\HK_{\Theta_{i-1}}]\otimes P_i\otimes K[\HK_{\Theta_{i+1}}]\otimes\cdots\otimes K[\HK_{\Theta_{m}}]$$
for maximal ideals $P_i$ of $K[\HK_{\Theta_{i}}]$. \\
Conversely, every ideal of $K[\HK_{\Theta'}]$ of this form is
maximal.
\end{theorem}
Note that maximal ideals of $K[C_j]$ have been characterized in
the previous sections of the paper and if $\Theta_{i}$ is a singleton, then $K[\HK_{\Theta_{i}}]\cong K\oplus K$, has two obvious
maximal ideals.

\begin{proof}
As we know that $K[\HK_{\Theta}]/{\mathcal
J}(K[\HK_{\Theta}])\cong K[\HK_{\Theta'}]$ it is clear that there
exists a one-to-one correspondence between maximal ideals of
$K[\HK_{\Theta}]$ and those of $K[\HK_{\Theta'}]$.

So, it remains to find all maximal ideals in $K[\HK_{\Theta'}]$.
Without loss of generality we can assume that $\Theta'$ has only
two connected components, that is $K[\HK_{\Theta'}]=R_1\otimes
R_2$, where $R_i=K[\HK_{\Theta_i}]$ are isomorphic to either
$K[C_j]$ or $K \oplus K$. Let $P$ be a maximal ideal of
$K[\HK_{\Theta'}]$ and $\pi: R_1\otimes R_2\rightarrow (R_1\otimes
R_2) / P$ the natural projection. Since $R_1\otimes R_2$ is a
PI-algebra over an algebraically closed field, from Kaplansky theorem it follows that
$(R_1\otimes R_2) / P\simeq M_{r}(K)$ for some $r\geqslant 1$. Denote  $\bar{R_1}=\pi(R_1\otimes
K)$ and $\bar{R_2}=\pi(K\otimes R_2)$.
Recall that the Jacobson
radical of a finitely generated PI-algebra is nilpotent,
Theorem~6.3.39 in \cite{Rowen-ring}. We claim that the algebras
$\bar{R_i}$ are semisimple. Indeed, let $N_1/P$ be a nilpotent
ideal in $\bar{R_1}$. Then $N_1(1\otimes R_2)$ is a nilpotent ideal
in $K[\HK_{\Theta'}]$, as $N_1(1\otimes R_2)=(1\otimes R_2)N_1$.
It follows that $N_1=0$ and $\bar{R_1}$ is semisimple, as it is finite dimensional. A symmetric
argument shows that $\bar{R_2}$ is also semisimple.
If $\bar{R_i}$ was not simple, then it would contain a non-trivial central idempotent. This idempotent would be
then also central in  $K[\HK_{\Theta'}] / P $, a contradiction. Therefore
$\bar{R_i}\simeq M_{r_i}(K)$ for $i = 1, 2$. Projection $\pi :
K[\HK_{\Theta'}] \rightarrow K[\HK_{\Theta'}] / P $ factors
through $ K[\HK_{\Theta'}]
\xrightarrow{\bar{\pi}_1}\bar{R_1}\otimes \bar{R_2}
\xrightarrow{\bar{\pi}_2}  K[\HK_{\Theta'}] / P$. Moreover, as
$\bar{R_1}\otimes \bar{R_2}\simeq M_{r_1+r_2}(K)$, $\bar{\pi}_2$
is an isomorphism. Then we can assume that $\pi: K[\HK_{\Theta'}]
\twoheadrightarrow M_{r}(K)$ where $r=r_1+r_2$.
If we denote $\pi_{1}=\pi|_{
R_1\otimes K}$ and $\pi_{2}=\pi|_{K\otimes R_2}$, then
$\ker(\pi_i)=P_i$ are maximal ideals in $R_i$.
It can be easily checked that $P=P_1\otimes R_2+R_1\otimes P_2$.

Conversely, if $P=P_1\otimes R_2+R_1\otimes P_2$ for maximal ideals
$P_1\triangleleft R_1$, $P_2\triangleleft R_2$, then $(R_1\otimes R_2) / P \simeq R_1/P_1 \otimes R_2/P_2$.
As $R_i$ are finitely generated PI-algebras it follows that
$ R_1/P_1 \otimes R_2/P_2\simeq M_{r_1}(K)\otimes M_{r_2}(K)\simeq M_{r_1+r_2}(K)$.
Therefore $P$ is indeed a maximal ideal in $K[HK_{\Theta'}]$.
\end{proof}

\noindent {\bf Acknowledgments.} I am very grateful to Jan Okni\'nski for suggesting this research topic, his continuous support and insightful comments on the earlier versions of the paper. This work was supported by grant 2016/23/B/ST1/01045 of the National Science Centre (Poland).

\vspace{40pt}

\begin{tabular}{lll}
      Magdalena Wiertel \\
    \texttt{M.Wiertel@mimuw.edu.pl} \\
    & & \\
    Institute of Mathematics & & \\
    University of Warsaw & & \\
    Banacha 2 & & \\
    02-097 Warsaw, Poland & &
\end{tabular}

\end{document}